\documentclass[reqno]{amsart}
\usepackage{bbold}
\usepackage{amsmath}
\usepackage{amssymb}
\usepackage{amsfonts}

\newcommand{\abs}[1]{\vert #1 \vert}

\newcommand{\norm}[1]{\left\Vert #1 \right\Vert}

\newcommand{\C}{\mathbb{C}}

\newcommand{\R}{\mathbb{R}}

\DeclareMathOperator{\im}{Im}
\DeclareMathOperator{\re}{Re}

\newtheorem{theorem}{Theorem}

\newtheorem{lemma}{Lemma}

\theoremstyle{definition}

\theoremstyle{remark}

\title[Ill-posedness for Maxwell-Dirac]{Ill-posedness for the Maxwell-Dirac system below charge in space dimension three and lower}

\author[S.~Selberg]{Sigmund Selberg}
\author[A.~Tesfahun]{Achenef Tesfahun}

\address{Department of Mathematics, University of Bergen, PO Box 7803, 5020 Bergen, Norway}
\email{Sigmund.Selberg@uib.no}

\email{Achenef.Temesgen@uib.no}


\thanks{This work was partially supported by the project Pure Mathematics in Norway, funded by the Trond Mohn Foundation}

\begin{document}

\begin{abstract}

The Maxwell-Dirac system describes the interaction of an electron with its self-induced electromagnetic field. In space dimension $d=3$ the system is charge-critical, that is, $L^2$-critical for the spinor with respect to scaling, and local well-posedness is known almost down to the critical regularity. In the charge-subcritical dimensions $d=1,2$, global well-posedness is known in the charge class. Here we prove that these results are sharp (or almost sharp, if $d=3$), by demonstrating ill-posedness below the charge regularity. In fact, for $d \le 3$ we exhibit a spinor datum belonging to $H^s(\R^d)$ for $s<0$, and to $L^p(\R^d)$ for $1 \le p < 2$, but not to $L^2(\R^d)$, which does not admit any local solution that can be approximated by smooth solutions in a reasonable sense.
\end{abstract}

\maketitle

\section{Introduction}

We consider the Maxwell-Dirac system
\begin{equation}\label{MD}
\left\{
\begin{aligned}
  &(-i\gamma^\mu \partial_\mu + M) \psi = A_\mu \gamma^\mu \psi,
  \\
  &\square A_\mu = \overline{\psi} \gamma_\mu \psi,
\end{aligned}
\right.
\end{equation}
on the Minkowski space-time $\R^{1+d}$ for space dimensions $d \le 3$. This fundamental model from relativistic field theory describes the interaction of an electron with its self-induced electromagnetic field. Our interest here is in the Cauchy problem with prescribed initial data at time $t=0$,
\begin{equation}\label{Data}
  \psi(0,x) = \psi_0(x), \quad A_\mu(0,x) = a_\mu(x), \quad \partial_t A_\mu(0,x) = b_\mu(x),
\end{equation}
and the question of local or global solvability, which has received some attention in recent years; see \cite{Chadam1973, Delgado1978, Bachelot2006, Huh2010, Okamoto2013, Zhang2014} for the case of one space dimension and \cite{Gross1966, Georgiev1991, Bournaveas1996, Esteban1996, Flato1997, Psarelli2005, Selberg2010b, Selberg2011, Gavrus2017} for higher dimensions, and the references therein.

The unknowns are the spinor field $\psi=\psi(t,x)$, taking values in $\C^N$ ($N=2$ for $d=1,2$; $N=4$ for $d=3$), and the real-valued potentials $A_\mu=A_\mu(t,x)$, $\mu=0,1,\dots,d$. $M \ge 0$ is the mass.

The equations are written in covariant form on $\R^{1+d} = \R_t \times \R_x^d$ with the Minkowski metric $(g^{\mu\nu}) = \mathrm{diag}(1,-1,\dots,-1)$ and coordinates $(x_\mu)$, where $x_0=t$ is the time and $x=(x_1,\dots,x_d)$ is the spatial position. Greek indices range over $0,1,\dots,d$, latin indices over $1,\dots,d$, and repeated upper and lower indices are implicitly summed over these ranges. We write $\partial_\mu = \frac{\partial}{\partial x_\mu}$, so $\partial_0=\partial_t$ is the time derivative, $\nabla = (\partial_1,\dots,\partial_d)$ is the spatial gradient, and $\square = \partial^\mu \partial_\mu = \partial_t^2-\Delta$ is the D'Alembertian. The $N \times N$ Dirac matrices $\gamma^\mu$ are required to satisfy
\begin{equation}\label{Matrices}
  \gamma^\mu \gamma^\nu + \gamma^\nu \gamma^\mu = 2 g^{\mu\nu} I,
  \qquad
  (\gamma^0)^* = \gamma^0, \qquad (\gamma^j)^* = - \gamma^j.
\end{equation}
We denote by $\psi^*$ the complex conjugate transpose, and write $\overline \psi = \psi^*\gamma^0$.

Key features of the Maxwell-Dirac system are the gauge invariance, the scaling invariance and the conservation laws, which we now recall.

Firstly, there is a $\mathrm{U}(1)$ gauge invariance
\[
  \psi \longrightarrow e^{i\chi}\psi, \qquad A_\mu \longrightarrow A_\mu + \partial_\mu \chi,
\]
for any real valued $\chi(t,x)$. This implies gauge freedom, allowing to specify a gauge condition on the potentials. The particular form \eqref{MD} of the Maxwell-Dirac system appears when the Lorenz gauge condition $\partial^\mu A_\mu = 0$ is imposed, that is,
\begin{equation}\label{Lorenz}
  \partial_t A_0 = \nabla \cdot \mathbf A,
\end{equation}
where $\mathbf A = (A_1,\dots,A_d)$. Since this gauge condition reduces to certain constraints on the data \eqref{Data}, we did not include it in \eqref{MD}. In addition to the obvious constraint, there is a constraint coming from the Gauss law (implied by \eqref{Lorenz} and the second equation in \eqref{MD})
\[
  \nabla \cdot \mathbf E = \abs{\psi}^2,
\]
where $\mathbf E = \nabla A_0 - \partial_t \mathbf A$ is the electric field. Thus, the data constraints are
\begin{equation}\label{Constraints}
  b_0 = \partial^j a_j, \qquad \partial^j (\partial_j a_0 - b_j) = \abs{\psi_0}^2.
\end{equation}
If these are satisfied (in some ball), then a solution of \eqref{MD}, \eqref{Data} will also satisfy the Lorenz gauge condition \eqref{Lorenz} (in the cone of dependence over the ball).

Secondly, the system is invariant under the rescaling, in the case $M=0$,
\[
  \psi(t,x) \longrightarrow \lambda^{3/2} \psi(\lambda t, \lambda x), \qquad A_\mu(t,x) \longrightarrow \lambda A_\mu(\lambda t, \lambda x) \qquad (\lambda > 0).
\]
For Sobolev data $(\psi_0,a_\mu,b_\mu) \in H^s(\R^d) \times H^r(\R^d) \times H^{r-1}(\R^d)$, the scale-invariant regularity (for the homogeneous Sobolev norms, to be precise) is $s=s_c(d)=\frac{d-3}{2}$ and $r=r_c(d)=\frac{d-2}{2}$. By the usual heuristic one does not expect well-posedness below this regularity.

Thirdly, we consider conservation laws. While the Maxwell-Dirac system does have a conserved energy, which is roughly speaking at the level of $H^{1/2}$ for the spinor, this energy does not have a definite sign, so it is difficult to see how to make use of it to prove global existence. On the other hand, one has the conservation of charge
\begin{equation}\label{Charge}
  \int_{\R^d} \abs{\psi(t,x)}^2 \, dx = \int_{\R^d} \abs{\psi(0,x)}^2 \, dx,
\end{equation}
which plays a key role in all the known global existence results for large data. We will refer to solutions at this regularity, that is, with $t \mapsto \psi(t,\cdot)$ a continuous map into $L^2(\R^d)$, as \emph{charge class solutions}. It should be noted that the charge regularity $s=0$ coincides with the scaling-critical regularity $s_c(d) = \frac{d-3}{2}$ when $d=3$. Thus, the Maxwell-Dirac system is charge-critical in three space dimensions, and charge-subcritical in dimensions $d=1,2$.

The first global result for \eqref{MD}, \eqref{Data} was obtained by Chadam \cite{Chadam1973}, in one space dimension, for data $(\psi_0,a_\mu,b_\mu) \in H^1(\R) \times H^1(\R) \times L^2(\R)$. Chadam first proved local existence and uniqueness, and was able to extend the solution globally by proving an a priori bound on the $H^1(\R) \times H^1(\R) \times L^2(\R)$ norm of the solution via a clever boot-strap argument making use of the conservation of charge \eqref{Charge}. But to be able to prove global existence with a more direct use of the conservation of charge, in any dimension, a natural strategy is to try to prove local existence of charge class solutions.\footnote{However, it should be noted that such a result does not immediately imply global existence via the conservation of charge, since one also needs a priori estimates for the potentials.}

We proceed to recall what is known about local and global well-posedness in the charge class.

Starting with one space dimension, we note that Bournaveas \cite{Bournaveas2000} proved global charge-class existence for the related Dirac-Klein-Gordon system, but the argument relies on a null structure in Dirac-Klein-Gordon which is not present in Maxwell-Dirac. Bachelot \cite{Bachelot2006} gave another proof that does not rely on null structure and applies also to Maxwell-Dirac; similar results have been obtained in \cite{Huh2010, Zhang2014, Selberg2018}.

In the charge-critical three space-dimensional case, local well-posedness remains an open question in the charge class, but has been proved almost down to that regularity by D'Ancona, Foschi and Selberg \cite{Selberg2010b}; see also \cite{Bournaveas1996,Gross1966} for earlier local results at higher regularity, and \cite{Georgiev1991,Flato1997,Psarelli2005} for small-data global results. The existence of stationary solutions was proved in \cite{Esteban1996}.

In two space dimensions, global well-posedness in the charge class was proved by D'Ancona and Selberg \cite{Selberg2011}.

To summarise, in the charge-subcritical dimensions $d=1,2$, there is global well-posedness in the charge class, and in the charge-critical dimension $d=3$, local well-posedness holds almost down to the charge regularity. Our aim here is to show that these results are sharp (or almost sharp, for $d=3$), by proving ill-posedness below the charge regularity. This result is somewhat surprising in the subcritical cases, and in particular for $d=1$. Indeed, it should be noted that in dimensions $d=2,3$, the proof of local existence at or near the charge regularity is quite involved and requires a subtle null structure that was uncovered in \cite{Selberg2010b}. By contrast, the proof in the case $d=1$ is elementary (see section \ref{WPproof}) and does not require this null structure. It was therefore expected that, by exploiting the latter, one should be able to go below the charge. But our result shows that this is not possible, and this means that the null structure is not helpful in the case $d=1$.

We remark that our proof of ill-posedness works also in dimensions $d \ge 4$, but then the critical regularity is above the charge, so this is not really of much interest. In dimensions $d \ge 4$, global existence and modified scattering for data with small scaling-critical norm has been proved in \cite{Gavrus2017}.

We now state our main results.

\section{Main results}

The following notation is used. For $1 \le p \le \infty$, $L^p(\R^d)$ denotes the standard Lebesgue space. For $s \in \R$, $H^s(\R^d)$ is the Sobolev space $(1-\Delta)^{s/2} L^2(\R^d)$. For an open set $U$ in $\R^d$ or $\R_t \times \R_x^d$, $\mathcal D'(U)$ is the space of distributions on $U$. We write
\begin{align*}
  X_0 &= \text{$H^s(\R^d)$ for some $s < 0$, or $L^p(\R^d)$ for some $1 \le p < 2$}.
  \\
  B &= \text{the open unit ball in $\R^d$, centred at the origin}.
  \\
  K &= \text{the cone of dependence over $B$}.
  \\
  K_T &= \text{$K \cap \left([0,T] \times \R^d \right)$, for $T > 0$}.
\end{align*}
Thus, $K = \{ (t,x) \in \R \times \R^d \colon 0 \le t < 1,\; \abs{x} < 1-t \}$. The interior of the truncated cone $K_T$ will be denoted $\mathrm{Int}(K_T)$.

We will use the following facts concerning $C^\infty$ solutions of \eqref{MD}, which follow from the general theory for semilinear wave equations. Assume we are given data \eqref{Data} belonging to $C^\infty(\R^d)$. Then there exists a corresponding $C^\infty$ solution $(\psi,A_\mu)$ of \eqref{MD} on an open subset $U$ of $[0,\infty) \times \R^d$ containing the Cauchy hypersurface $\{0\} \times \R^d$. Moreover, we may assume that $U$ is \emph{causal}, in the sense that for every point $(t,x)$ in $U$, the cone of dependence $K^{(t,x)}$, with vertex $(t,x)$ and base in $\{0\} \times \R^d$, is contained in $U$. The solution in the cone $K^{(t,x)}$ is uniquely determined by the data in the base of the cone. By the uniqueness, and since the union of two causal sets is again causal, there exists a maximal solution of the type described above, and we call this the \emph{maximal $C^\infty$ forward evolution} of the given data. 

In the first version of our ill-posedness result, we take vanishing data for the potentials.

\begin{theorem}[Ill-posedness I]\label{MainThm1}
In space dimensions $d \le 3$, the Cauchy problem \eqref{MD}, \eqref{Data} is ill posed for data
\[
  \psi_0 \in X_0, \qquad a_\mu = b_\mu = 0 \qquad (\mu=0,\dots,d).
\]
More precisely, there exists $\psi_0^{\mathrm{bad}} \in X_0 \setminus L^2(\R^d)$ such that for any $T > 0$ and any neighbourhood $\Omega_0$ of $\psi_0^{\mathrm{bad}}$ in $X_0$, there fails to exist a continuous map
\[
  S \colon \Omega_0 \longrightarrow \mathcal D'\left( \mathrm{Int}(K_T) \right),
  \qquad \psi_0 \longmapsto S[\psi_0] = (\psi,A_\mu),
\]
with the property that if $\psi_0 \in \Omega_0 \cap C_c^\infty(\R^d)$, then $S[\psi_0]$ is $C^\infty$ in $K_T$ and solves \eqref{MD} there, with intial data $\psi_0$ and $a_\mu=b_\mu=0$ in $B$.
\end{theorem}

This result applies to \eqref{MD}, \eqref{Data} without regard to the data constraints \eqref{Constraints}, which of course are not compatible with the assumption $a_\mu=b_\mu=0$. We next state an alternative version of the result, which allows to take into account the constraints. In fact, Theorem \ref{MainThm1} is an immediate consequence of the following more precise result.

\begin{theorem}[Ill-posedness II]\label{MainThm2}
Let $d \le 3$. There exist $\psi_0^{\mathrm{bad}} \in X_0 \setminus L^2(\R^d)$ and $\psi_{0,\varepsilon}, a_{\mu,\varepsilon}, b_{\mu,\varepsilon} \in C_c^\infty(\R^d)$ for each $\varepsilon > 0$, such that
\begin{itemize} 
\item[(i)] $\psi_{0,\varepsilon} \to \psi_0^{\mathrm{bad}}$ in $X_0$ as $\varepsilon \to 0$.
\item[(ii)] The maximal $C^\infty$ forward evolution $(\psi_\varepsilon,A_{\mu,\varepsilon})$ of the data $(\psi_{0,\varepsilon}, a_{\mu,\varepsilon}, b_{\mu,\varepsilon})$ exists throughout the cone $K$.
\item[(iii)] There exists $T > 0$ such that, as $\varepsilon \to 0$, $A_{0,\varepsilon}(t,x) \to \infty$ uniformly in any compact subset of $K_T \cap \{ (t,x) \colon \abs{x} < t \}$.
\end{itemize}
Moreover, we can choose the $a_{\mu,\varepsilon}, b_{\mu,\varepsilon}$ so that either
\begin{equation}\label{ZeroData}
  a_{\mu,\varepsilon} = b_{\mu,\varepsilon} = 0 \qquad \text{for $\mu = 0,\dots,d$},
\end{equation}
or
\begin{equation}\label{ConstrainedData}
  b_{0,\varepsilon} = \sum_{j=1}^d \partial_j a_{j,\varepsilon},
  \qquad
  \sum_{j=1}^d \partial_j \left( \partial_j a_{0,\varepsilon} - b_{j,\varepsilon} \right) = \abs{\psi_{0,\varepsilon}}^2 \qquad \text{in $B$}.
\end{equation}
\end{theorem}

Here, if we choose the alternative \eqref{ConstrainedData}, then $a_{\mu,\varepsilon}, b_{\mu,\varepsilon}$ do not have limits in the sense of distributions on $B$ as $\varepsilon \to 0$. This is not a deficiency of our construction, but is necessarily so, as our next result shows. The following theorem essentially says that the Gauss law for the initial data is ill posed when we are below the charge regularity.

\begin{theorem}[Ill-posedness of constraints]\label{MainThm3}
There exists $\psi_0^{\mathrm{bad}} \in X_0 \setminus L^2(\R^d)$ such that for any neighbourhood $\Omega_0$ of $\psi_0^{\mathrm{bad}}$ in $X_0$, there do not exist continuous maps
\[
  I_\mu, J_\mu \colon \Omega_0 \longrightarrow \mathcal D'(B)
\]
with the property that if $\psi_0 \in \Omega_0 \cap C_c^\infty(\R^d)$, then
\[
  a_\mu := I_\mu[\psi_0], \qquad b_\mu := J_\mu[\psi_0] \qquad (\mu=0,\dots,d)
\]
satisfy the constraint equations \eqref{Constraints} in $B$.
\end{theorem}

We conclude this section with a brief outline of the key steps in the proof of Theorem \ref{MainThm2}.

\medskip
\emph{Step 1.} We prove global well-posedness in the charge class for \eqref{MD}, \eqref{Data} in the case where the data only depend on a single coordinate, say $x_1$.

\medskip
\emph{Step 2.} To define the data $\psi_{0,\varepsilon}, a_{\mu,\varepsilon}, b_{\mu,\varepsilon} \in C_c^\infty(\R^d)$, we start with functions of $x_1$ and cut off smoothly outside the unit ball $B$. The corresponding maximal $C^\infty$ forward evolution $(\psi_\varepsilon,A_{\mu,\varepsilon})$ exists in the entire cone $K$, by Step 1, and depends only on $t$ and $x_1$ there.

\medskip
\emph{Step 3.} Using a null form estimate and a boot-strap argument we prove that there exists $T > 0$ such that $A_{j,\varepsilon}$, $j=2,3$, are uniformly bounded in $K_T$. A further boot-strap argument then yields a lower bound on $\abs{\psi_\varepsilon}$ in
$
  K_T \cap \{ (t,x) \colon 0 < t < x_1 \}.
$

\medskip
\emph{Step 4.}
Letting $\varepsilon \to 0$, we show that $A_{0,\varepsilon}(t,x) \to \infty$ uniformly in any compact subset of
$
  K_T \cap \{ (t,x) \colon \abs{x} < t \},
$
completing the proof of Theorem \ref{MainThm2}.
In fact, we prove this in the larger set $K_T \cap \{ (t,x) \colon \abs{x_1} < t \}$.
\medskip

The remainder of this paper is organised as follows. In section \ref{WPsection} we state the well-posedness result (Step 1), whose elementary proof is deferred until section \ref{WPproof}. In section \ref{MatricesSection}, we choose a particular representation of the Dirac matrices in dimensions $d \le 3$, write out the Maxwell-Dirac system in terms of the components of the spinor, and prove a null form estimate in one space dimension. In section \ref{DataSection} we specify the data (Step 2), and section \ref{IPsection} contains the proof of ill-posedness (Steps 3 and 4).

\section{Well-posedness for one-dimensional data}\label{WPsection}

We start by stating the result described in Step 1, the well-posedness in the case where the data only depend on the single coordinate $x_1$:
\begin{equation}\label{1dData}
  \psi(0,x) = \psi_0(x_1), \quad A_\mu(0,x) = a_\mu(x_1), \quad \partial_t A_\mu(0,x) = b_\mu(x_1).
\end{equation}
Then the solution of \eqref{MD} will depend only on $t$ and $x_1$. Indeed, if $(\psi,A_\mu)$ does not depend on $x_2,\dots,x_d$, then \eqref{MD} is equivalent to
\begin{equation}\label{MD'}
\left\{
\begin{aligned}
  (- i\gamma^0 \partial_t - i\gamma^1 \partial_1 + M) \psi &= \left( A_0 \gamma^0 + A_1 \gamma^1 + \dots + A_d \gamma^d \right) \psi,
  \\
  (\partial_t^2 - \partial_1^2) A_0 &= \psi^*  \psi,
  \\
  (\partial_t^2 - \partial_1^2) A_1 &= - \psi^* \gamma^0 \gamma^1 \psi,
  \\
  &\ \,\vdots
  \\
  (\partial_t^2 - \partial_1^2) A_d &= - \psi^* \gamma^0 \gamma^d \psi,
\end{aligned}
\right.
\end{equation}
and this is the system we will solve, with the initial condition \eqref{1dData}.

There is conservation of charge, for sufficiently regular solutions:
\begin{equation}\label{1dCharge}
  \int_{\R} \abs{\psi(t,x_1)}^2 \, dx_1 = \int_{\R^d} \abs{\psi(0,x_1)}^2 \, dx_1.
\end{equation}
Indeed, premultiplying the Dirac equation in \eqref{MD'} by $i\overline\psi=i\psi^*\gamma^0$, taking real parts, and using the fact that $M$ and the $A_\mu$ are real, and that $\gamma^0$ and $\gamma^0\gamma^j$ are hermitian, one obtains the conservation law $\partial_t \rho + \partial_1 j = 0$, where $\rho = \psi^*\psi = \abs{\psi}^2$ and $j = \psi^*\gamma^0\gamma^1\psi$. Integration then gives \eqref{1dCharge}.

We now state the global well-posedness result in the charge class. The $a_\mu$, $\mu=0,\dots,d$, will be taken in the space $AC(\R)$ with norm
\[
  \norm{f}_{AC(\R)} = \norm{f}_{L^\infty(\R)} + \norm{f'}_{L^1(\R)}.
\]
Thus, $AC(\R)$ is the space of absolutely continuous functions $f \colon \R \to \C$ with bounded variation (cf.\ Corollary 3.33 in \cite{Folland1999}), and $AC_{\mathrm{loc}}(\R)$ is the space of locally absolutely continuous functions.

\begin{theorem}\label{WPThm}
In any space dimension $d$, the Maxwell-Dirac system \eqref{MD} is globally well-posed for one-dimensional data \eqref{1dData} with the regularity
\[
  (\psi_0,a,b) \in \mathfrak X_0 := L^2(\R;\C^N) \times AC(\R;\R^{d+1}) \times L^1(\R;\R^{d+1}),
\]
where $a = (a_0,\dots,a_d)$ and $b=(b_0,\dots,b_d)$. That is, for any $T > 0$, there is a unique solution
\[
  (\psi,A,\partial_t A) \in C([0,T];\mathfrak X_0), \qquad A=(A_0,\dots,A_d),
\]
depending only on $t$ and $x_1$. The solution has the following properties:
\begin{itemize}
\item[(i)] The data-to-solution map is continuous from $\mathfrak X_0$ to $C([0,T];\mathfrak X_0)$.
\item[(ii)] Higher regularity persists. That is, if $J \in \mathbb N$ and $\partial_1^j (\psi_0,a_\mu,b_\mu) \in \mathfrak X_0$ for $j \le J$, then $\partial_t^j \partial_1^k (\psi,A_\mu,\partial_t A_\mu) \in C([0,T];\mathfrak X_0)$ for $j+k \le J$.
\item[(iii)] If the data are $C^\infty$, then so is the solution.
\item[(iv)] The conservation of charge \eqref{1dCharge} holds.
\item[(v)] If the data constraints \eqref{Constraints} are satisfied for $x_1$ in an interval $I$, then the Lorenz gauge condition $\partial_t A_0 = \partial_1 A_1$ is satisfied in the cone of dependence over $I$.
\end{itemize}
\end{theorem}

In particular, taking $d=1$, this result provides an alternative to the charge-class results from \cite{Bachelot2006,Zhang2014}, with a stronger form of well-posedness and at the same time a much simpler proof. The elementary proof is given in section \ref{WPproof}. We use iteration to prove local existence, and to close the estimates we only rely on the energy inequality for the Dirac equation and an estimate for the wave equation deduced from the D'Alembert representation.

\section{The Dirac matrices and a null form estimate}\label{MatricesSection}

In this section we specify our choice of the Dirac matrices, in dimensions $d \le 3$. We do this in such a way that the Dirac equation in \eqref{MD'}, when written in terms of the spinor components, has a form which makes it easy to work with. Recall that that the spinor has $N=2$ components in space dimensions $d=1,2$, and $N=4$ components when $d=3$. We write
\[
  \psi = \begin{pmatrix} u \\ v \end{pmatrix},
\]
where $u,v$ are $\C$-valued for $d=1,2$ and $\C^2$-valued for $d=3$.

\subsection{Space dimension $d=1$}

We choose
\[
  \gamma^0 =
  \left( \begin{matrix}
    0 & 1  \\
    1 &0
  \end{matrix} \right),
  \qquad
  \gamma^1=
  \left( \begin{matrix}
    0 & -1  \\
    1 & 0
  \end{matrix} \right).
\]
Then \eqref{Matrices} is satisfied, and \eqref{MD'} becomes
\begin{equation}\label{MD1d}
\left\{
\begin{aligned}
  (\partial_t + \partial_x)u &= i (A_0 + A_1) u -iMv,
  \\
  (\partial_t - \partial_x)v &= i (A_0 - A_1) v -iMu,
  \\
  (\partial_t^2-\partial_x^2) A_0 &= \abs{u}^2 + \abs{v}^2,
  \\
  (\partial_t^2-\partial_x^2) A_1 &= - \abs{u}^2 + \abs{v}^2.
\end{aligned}
\right.
\end{equation}
Since $A_0$, $A_1$ are real valued, the first two equations imply
\begin{equation}\label{MD1dSquared}
\left\{
\begin{aligned}
  (\partial_t + \partial_x)\abs{u}^2 &= -2M \im\left( \overline v u \right),
  \\
  (\partial_t - \partial_x)\abs{v}^2 &= 2M \im\left( \overline v u \right).
\end{aligned}
\right.
\end{equation}

\subsection{Dimension $d=2$}

We choose
\[
  \gamma^0 =
  \left( \begin{matrix}
    0 & 1  \\
    1 &0
  \end{matrix} \right),
  \qquad
  \gamma^1=
  \left( \begin{matrix}
    0 & -1  \\
    1 & 0
  \end{matrix} \right),
  \\
  \qquad
  \gamma^2=
  \left( \begin{matrix}
    i & 0  \\
    0 & -i
  \end{matrix} \right)
\]
Then \eqref{Matrices} is satisfied, and \eqref{MD'} becomes, writing $x=x_1$ for simplicity,
\begin{equation}\label{MD2d}
\left\{
\begin{aligned}
  (\partial_t + \partial_x)u &= i (A_0 + A_1) u + A_2 v -iMv,
  \\
  (\partial_t - \partial_x)v &= i (A_0 - A_1) v - A_2 u -iMu,
  \\
  (\partial_t^2-\partial_x^2) A_0 &= \abs{u}^2 + \abs{v}^2,
  \\
  (\partial_t^2-\partial_x^2) A_1 &= - \abs{u}^2 + \abs{v}^2,
  \\
  (\partial_t^2-\partial_x^2) A_2 &= - 2 \im (u \overline v).
\end{aligned}
\right.
\end{equation}
Then we also have
\begin{equation}\label{MD2dSquared}
\left\{
\begin{aligned}
  (\partial_t + \partial_x)\abs{u}^2 &=  2A_2 \re \left( \overline v u \right)
  -2M \im\left( \overline v u \right),
  \\
  (\partial_t - \partial_x)\abs{v}^2 &= - 2A_2 \re \left( \overline v u \right) + 2M \im \left( \overline v u \right).
\end{aligned}
\right.
\end{equation}

\subsection{Dimension $d=3$}

The $4 \times 4$ Dirac matrices are, in $2 \times 2$ block form,
\[
  \gamma^0 =
  \left( \begin{matrix}
    0 & I  \\
    I &0
  \end{matrix} \right),
  \qquad
  \gamma^1=
  \left( \begin{matrix}
    0 & -I  \\
    I & 0
  \end{matrix} \right),
  \\
  \qquad
  \gamma^2=
  \left( \begin{matrix}
    \rho & 0  \\
    0 & -\rho
  \end{matrix} \right),
  \\
  \\
  \qquad
  \gamma^3=
  \left( \begin{matrix}
    \kappa & 0  \\
    0 & -\kappa
  \end{matrix} \right),
\]
where $I$ is the $2 \times 2$ identity matrix and $\rho$, $\kappa$ must satisfy
\[
  \rho^* = - \rho, \qquad \rho^2 = -I, \qquad \kappa^* = -\kappa, \qquad \kappa^2 = -I, \qquad \rho\kappa + \kappa\rho = 0.
\]
Then \eqref{Matrices} is satisfied. For example, we can choose
\[
  \rho =
  \left( \begin{matrix}
    0 & -1  \\
    1 & 0
  \end{matrix} \right),
  \\
  \qquad
  \kappa =
  \left( \begin{matrix}
    i & 0  \\
    0 & -i
  \end{matrix} \right).
\]
Then \eqref{MD'} reads (with $x=x_1$)
\begin{equation}\label{MD3d}
\left\{
\begin{aligned}
  (\partial_t + \partial_x)u &= i (A_0 + A_1) u - i A_2 \rho v - i A_3 \kappa v -iMv,
  \\
  (\partial_t - \partial_x)v &= i (A_0 - A_1) v + i A_2 \rho u + i A_3 \kappa u -iMu,
  \\
  (\partial_t^2-\partial_x^2) A_0 &= \abs{u}^2 + \abs{v}^2,
  \\
  (\partial_t^2-\partial_x^2) A_1 &= - \abs{u}^2 + \abs{v}^2,
  \\
  (\partial_t^2-\partial_x^2) A_2 &= - 2 \re (v^* \rho u),
  \\
  (\partial_t^2-\partial_x^2) A_3 &= - 2 \re (v^* \kappa u),
\end{aligned}
\right.
\end{equation}
where $u, v$ are now $\C^2$-valued. Then also
\begin{equation}\label{MD3dSquared}
\left\{
\begin{aligned}
  (\partial_t + \partial_x)\abs{u}^2 &=  2A_2 \im \left( v^* \rho u \right)
  + 2A_3 \im \left( v^* \kappa u \right)
  - 2M \im\left( v^* u \right),
  \\
  (\partial_t - \partial_x)\abs{v}^2 &= - 2A_2 \im \left( v^* \rho u \right)
  - 2A_3 \im \left( v^* \kappa u \right)
  + 2M \im\left( v^* u \right).
\end{aligned}
\right.
\end{equation}

\subsection{A null form estimate}

When we move from $d=1$ to $d=2$ or $d=3$, the decisive difference is that we pick up the additional fields $A_2, A_3$. These fields will be better behaved than $A_0, A_1$, since the right hand sides of the corresponding equations in \eqref{MD2d} and \eqref{MD3d} are null forms: They contain a product of  $v^*$ and $u$, which propagate in transverse directions. This fact will be exploited through the following crucial estimate (which fails for $uu$ and $u \overline u$).

We use the following notation. For $x \in \R$ and $t > 0$, let $K^{(t,x)}$ denote the backward cone with vertex at $(t,x)$, that is,
\begin{equation}\label{1dCone}
  K^{(t,x)} = \left\{ (s,y) \in \R^2 \colon 0 < s < t, \;\; x-t+s < y < x+t-s \right\}.
\end{equation}

\begin{lemma}[Null form estimate]\label{NullLemma}
Consider a system of the form
\[
\begin{alignedat}{2}
  (\partial_t + \partial_x)u &= F(t,x), & \qquad u(0,x) &= f(x),
  \\
  (\partial_t - \partial_x)v &= G(t,x), & v(0,x) &= g(x),
\end{alignedat}
\]
where $x \in \R$, $t > 0$, and the functions are $\C$-valued. For the solution $(u,v)$ we have the estimate, for all $X \in \R$ and $T > 0$,
\[
  \iint_{K^{(T,X)}} \abs{uv} \, dx \, dt
  \\
  \le
  \left( \norm{f}_{L^1} + \int_0^T \norm{F(t)}_{L^1} \, dt \right)
  \left( \norm{g}_{L^1} + \int_0^T \norm{G(t)}_{L^1} \, dt \right).
\]
\end{lemma}

\begin{proof}
Integrating, we have
\begin{align*}
  u(t,x) &= f(x-t) + \int_0^t F(s,x-t+s) \, ds,
  \\
  v(t,x) &= g(x+t) + \int_0^t G(s,x+t-s) \, ds.
\end{align*}
Taking absolute values, we see that for $0 \le t \le T$,
\begin{align*}
  \abs{u(t,x)} &\le \mu(x-t) := \abs{f(x-t)} + \int_0^T \abs{F(s,x-t+s)} \, ds,
  \\
  \abs{v(t,x)} &\le \nu(x+t) := \abs{g(x+t)} + \int_0^T \abs{G(s,x+t-s)} \, ds.
\end{align*}
By Fubini's theorem it is then obvious that
\[
  \iint_{K^{(T,X)}} \abs{uv} \, dx \, dt
  \le
  \norm{\mu}_{L^1(\R)} \norm{\nu}_{L^1(\R)}.
\]
But
\[
  \norm{\mu}_{L^1(\R)}
  \le \norm{f}_{L^1(\R)} + \int_0^T \norm{F(t)}_{L^1(\R)} \, dt,
\]
and similarly for $\nu$, so we get the desired estimate.
\end{proof}

\section{Data for ill-posedness}\label{DataSection}

In this section we specify the data that are used to prove Theorem \ref{MainThm2}.

Choose a cut-off $\chi \in C_c^\infty(\R)$ such that $\chi=1$ on $[-1,1]$. Let $\varepsilon > 0$. For the spinor datum and its approximations, which are $\C^N$-valued, we then take
\begin{equation}\label{SpinorData}
  \psi_0^{\mathrm{bad}}(x) = \chi(x_1) \cdots \chi(x_d) \begin{pmatrix} f(x_1) \\ 0 \\ \vdots \\ 0 \end{pmatrix},
  \quad
  \psi_{0,\varepsilon}(x) = \chi(x_1) \cdots \chi(x_d) \begin{pmatrix} f_\varepsilon(x_1) \\ 0 \\ \vdots \\ 0 \end{pmatrix},
\end{equation}
where
\begin{equation}\label{SpinorData2}
  f(x_1) =  \frac{1}{\abs{x_1}^{1/2}},
  \qquad
  f_\varepsilon(x_1) =  \frac{1}{(\varepsilon^2+x_1^2)^{1/4}}.
\end{equation}
Thus, $\chi f_\varepsilon \in C_c^\infty(\R)$, and for $1 \le p < 2$ we have $\chi f \in L^p(\R) \setminus L^2(\R)$ and
\[
  \lim_{\varepsilon \to 0} \norm{ \chi f_\varepsilon - \chi f }_{L^p(\R)} = 0.
\]
By the Hardy-Littlewood-Sobolev inequality,\footnote{We use the inequality $\norm{g}_{H^s(\R)} \le C \norm{g}_{L^p(\R)}$, valid for $s = 1/2-1/p$, $1 \le p < 2$.} we then conclude that $\chi f \in H^s(\R)$ for $s < 0$, and that
\[
  \lim_{\varepsilon \to 0} \norm{ \chi f_\varepsilon - \chi f }_{H^s(\R)} = 0.
\]
It follows that $\psi_{0,\varepsilon} \in C_c^\infty(\R^d)$, $\psi_0^{\mathrm{bad}} \in X_0 \setminus L^2(\R^d)$, and
\[
  \lim_{\varepsilon \to 0} \norm{ \psi_{0,\varepsilon} - \psi_0^{\mathrm{bad}} }_{X_0} = 0,
\]
where as before $X_0$ denotes either $H^s(\R^d)$, $s < 0$, or $L^p(\R^d)$, $1 \le p < 2$.

Next, we choose the data $a_{\mu,\varepsilon}, b_{\mu,\varepsilon} \in C_c^\infty(\R^d)$. The first alternative is to take vanishing data
\begin{equation}\label{abZero}
  a_{0,\varepsilon} = \dots = a_{d,\varepsilon} = 0,
  \qquad
  b_{0,\varepsilon} = \dots = b_{d,\varepsilon} = 0,
\end{equation}
as in \eqref{ZeroData}. The second alternative is to ensure that the constraints in \eqref{ConstrainedData} are satisfied. For this, we take all the data to vanish except $b_{1,\varepsilon}$, so the constraints reduce to
\[
  - \partial_1 b_{1,\varepsilon} = \abs{\psi_{0,\varepsilon}}^2 = \frac{1}{\sqrt{\varepsilon^2+x_1^2}} \quad \text{in $B$}.
\]
Integrating this, we obtain
\begin{equation}\label{abConstrained}
\left\{
\begin{aligned}
  &a_{0,\varepsilon} = \dots = a_{d,\varepsilon} = 0,
  \qquad
  b_{0,\varepsilon} = b_{2,\varepsilon} = \dots = b_{d,\varepsilon} = 0,
  \\
  &b_{1,\varepsilon}(x) = - \chi(x_1) \cdots \chi(x_d) \log\left(x_1 + \sqrt{\varepsilon^2 + x_1^2} \right),
\end{aligned}
\right.
\end{equation}
which satisfies \eqref{ConstrainedData}.
 
\section{Proof of ill-posedness}\label{IPsection}

We start by proving Theorem \ref{MainThm2}, which implies Theorem \ref{MainThm1}. Theorem \ref{MainThm3} is proved at the end of this section.

Let $d \le 3$, choose the Dirac matrices as in section \ref{MatricesSection}, and define the data
$
  (\psi_{0,\varepsilon},a_{\mu,\varepsilon},b_{\mu,\varepsilon}) \in C_c^\infty(\R^d)
$
by \eqref{SpinorData}, \eqref{SpinorData2}, and either \eqref{abZero} or \eqref{abConstrained}. Since the data depend only on $x_1$ in $B$, it follows from Theorem \ref{WPThm} that their maximal $C^\infty$ forward evolution
$
  (\psi_\varepsilon,A_{\mu,\varepsilon})
$
exists throughout the cone $K$ over $B$, and depends only on $t$ and $x_1$ there. Indeed, we can apply Theorem \ref{WPThm} with the data restricted to $x_2=\dots=x_d=0$.

We now claim that for $T > 0$ sufficiently small, the following holds for $\varepsilon > 0$:
\begin{equation}\label{Claim1}
  \abs{A_{j,\varepsilon}(t,x)} \le 1 \quad \text{in $K_T$, for $2 \le j \le d$},
\end{equation}
and
\begin{equation}\label{Claim2}
  \abs{\psi_\varepsilon(t,x)}^2 \ge \frac12 \abs{f_\varepsilon(x_1-t)}^2 \quad \text{in $K_T \cap \{ (t,x) \colon 0 < t < x_1 \}$}.
\end{equation}
Moreover,
\begin{equation}\label{Claim3}
  A_{0,\varepsilon}(t,x) \ge c(Q) \abs{\log\varepsilon}
  \quad \text{in any compact $Q \subset K_T \cap \{ (t,x) \colon \abs{x_1} < t \}$},
\end{equation}
for all sufficiently small $\varepsilon > 0$, and some constant $c(Q) > 0$ depending only on $Q$.

Once we have obtained \eqref{Claim3}, then Theorem \ref{MainThm2} is proved. The plan is now as follows: First, we prove that \eqref{Claim2} implies \eqref{Claim3}, then we prove \eqref{Claim1}, and finally we prove \eqref{Claim2}.

Since \eqref{Claim1}--\eqref{Claim3} are restricted to the cone $K$, where the solution depends only on $t$ and $x_1$, it suffices to prove them for $x_2=\dots=x_d=0$. For the remainder of this section we therefore restrict to $x_2=\dots=x_d=0$. The solution then exists for all $t \ge 0$ and $x_1 \in \R$, by Theorem \ref{WPThm}. To simplify the notation we also write $x=x_1$.

\subsection{Proof that $\eqref{Claim2} \implies \eqref{Claim3}$}

Since $(\partial_t^2 - \partial_x^2) A_{0,\varepsilon} = \abs{\psi_\varepsilon}^2$ with vanishing initial data, we have by d'Alembert's formula
\[
  A_{0,\varepsilon}(t,x) = \frac12 \iint_{K^{(t,x)}} \abs{\psi_\varepsilon}^2 \, dy \, ds
  =
  \frac12 \int_0^t \int_{x-t+s}^{x+t-s} \abs{\psi_\varepsilon(s,y)}^2 \, dy \, ds,
\]
with notation as in \eqref{1dCone}.
Take $\abs{x} < t < T \ll 1$ and restrict the integration to the cone $K^{(t,x)} \cap \{ (s,y) \colon s < y \}$. Assuming \eqref{Claim2} holds, we thus obtain
\begin{align*}
  A_{0,\varepsilon}(t,x) &\ge \frac12 \int_0^{\frac{x+t}{2}} \int_{s}^{x+t-s} \abs{\psi_\varepsilon(s,y)}^2 \, dy \, ds
  \\
  &\ge
  \frac14 \int_0^{\frac{x+t}{2}} \int_{s}^{x+t-s} \frac{1}{\sqrt{\varepsilon^2+(y-s)^2}} \, dy \, ds
  \\
  &\ge
  \frac14 \int_0^{\frac{x+t}{2}} \int_{s}^{x+t-s} \frac{1}{\varepsilon+y-s} \, dy \, ds
  \\
  &= \frac{x+t}{8} (-\log\varepsilon)
  + \frac18 (\varepsilon+x+t) \left( \log(\varepsilon+x+t) - 1 \right)
  - \frac12 \varepsilon (\log\varepsilon-1),
\end{align*}
and \eqref{Claim3} follows.

\subsection{Proof of \eqref{Claim1}} This is only relevant in dimensions $d = 2,3$. We show the proof for $d=2$, and comment on $d=3$ at the end.

Assuming now $d=2$, then the system is as in \eqref{MD2d}:
\begin{equation}\label{MD2dEps}
\left\{
\begin{aligned}
  (\partial_t + \partial_x)u_\varepsilon &= i (A_{0,\varepsilon} + A_{1,\varepsilon}) u_\varepsilon + A_{2,\varepsilon} v_\varepsilon -iMv_\varepsilon,
  \\
  (\partial_t - \partial_x)v_\varepsilon &= i (A_{0,\varepsilon} - A_{1,\varepsilon}) v_\varepsilon - A_{2,\varepsilon} u_\varepsilon -iMu_\varepsilon,
  \\
  (\partial_t^2-\partial_x^2) A_{0,\varepsilon} &= \abs{u_\varepsilon}^2 + \abs{v_\varepsilon}^2,
  \\
  (\partial_t^2-\partial_x^2) A_{1,\varepsilon} &= - \abs{u_\varepsilon}^2 + \abs{v_\varepsilon}^2,
  \\
  (\partial_t^2-\partial_x^2) A_{2,\varepsilon} &= - 2 \im (u_\varepsilon \overline{v_\varepsilon}),
\end{aligned}
\right.
\end{equation}
with data
\[
  u_\varepsilon(0,x) = \chi f_\varepsilon(x) = \frac{\chi(x)}{(\varepsilon^2+x^2)^{1/4}},
  \\
  \qquad v_\varepsilon(0,x) = 0,
\]
and either \eqref{abZero} or \eqref{abConstrained} (with $x=x_1$ and $x_2=0$). The solution exists globally and is $C^\infty$, by Theorem \ref{WPThm}.

We want to prove \eqref{Claim1}. This will follow if we can prove that for $T > 0$ sufficiently small,
\begin{equation}\label{Claim1'}
  \sup_{(t,x) \in [0,T] \times \R} \abs{A_{2,\varepsilon}(t,x)} \le 1
\end{equation}
for all $\varepsilon > 0$.

By d'Alembert's formula, since $a_{2,\varepsilon} = b_{2,\varepsilon} = 0$,
\begin{equation}\label{A2}
  A_{2,\varepsilon}(t,x)  =
  \iint_{K^{(t,x)}} \im (u_\varepsilon \overline{v_\varepsilon})(s,y) \, dy \, ds,
\end{equation}
were $K^{(t,x)}$ denotes the backward cone \eqref{1dCone}.

The idea is now to apply Lemma \ref{NullLemma} to the first two equations in \eqref{MD2dEps}. But first we need to integrate out the terms involving $(A_{0,\varepsilon} \pm A_{1,\varepsilon})$. Define $\phi_{+,\varepsilon}$, $\phi_{-,\varepsilon}$ by
\[
\begin{alignedat}{2}
  (\partial_t + \partial_x) \phi_{+,\varepsilon} &= A_{0,\varepsilon}+A_{1,\varepsilon}, & \qquad \phi_{+,\varepsilon}(0,x) &= 0,
  \\
  (\partial_t - \partial_x) \phi_{-,\varepsilon} &= A_{0,\varepsilon}-A_{1,\varepsilon}, & \qquad \phi_{-,\varepsilon}(0,x) &= 0,
\end{alignedat}
\]
that is,
\begin{align*}
  \phi_{+,\varepsilon}(t,x) &= \int_0^t (A_{0,\varepsilon}+A_{1,\varepsilon})(s,x-t+s) \, ds,
  \\
  \phi_{-,\varepsilon}(t,x) &= \int_0^t (A_{0,\varepsilon}-A_{1,\varepsilon})(s,x+t-s) \, ds.
\end{align*}
Then from the first two equations in \eqref{MD2dEps} we get
\begin{equation}\label{uvMod}
\begin{aligned}
  (\partial_t + \partial_x)(e^{-i\phi_{+,\varepsilon}} u_\varepsilon) &= e^{-i\phi_{+,\varepsilon}}[ A_{2,\varepsilon} v_\varepsilon -iMv_\varepsilon],
  \\
  (\partial_t - \partial_x)(e^{-i\phi_{-,\varepsilon}} v_\varepsilon) &= e^{-i\phi_{-,\varepsilon}}[- A_{2,\varepsilon} u_\varepsilon -iMu_\varepsilon],
\end{aligned}
\end{equation}
so by \eqref{A2} and Lemma \ref{NullLemma},
\begin{equation}\label{A2Est0}
\begin{aligned}
  \norm{A_{2,\varepsilon}(t)}_{L^\infty}
  &\le
  \sup_{x \in \R} \iint_{K^{(t,x)}} \abs{u_\varepsilon} \abs{v_\varepsilon} \, dy \, ds
  \\
  &=
  \sup_{x \in \R} \iint_{K^{(t,x)}} \abs{e^{-i\phi_{+,\varepsilon}} u_\varepsilon} \abs{e^{-i\phi_{-,\varepsilon}} v_\varepsilon} \, dy \, ds
  \\
  &\le
  \left( \norm{\chi f_\varepsilon}_{L^1} + \int_0^t (M + \norm{A_{2,\varepsilon}(s)}_{L^\infty})\norm{v_\varepsilon(s)}_{L^1} \, ds \right)
  \\
  &\quad \times \left( \int_0^t (M + \norm{A_{2,\varepsilon}(s)}_{L^\infty})\norm{u_\varepsilon(s)}_{L^1} \, ds \right).
\end{aligned}
\end{equation}
To control the $L^1$ norms of $u_\varepsilon(t)$ and $v_\varepsilon(t)$, we use again \eqref{uvMod}, which implies
\begin{align*}
  (e^{-i\phi_{+,\varepsilon}} u_\varepsilon)(t,x) &= \chi f_\varepsilon(x-t) + \int_0^t \left(e^{-i\phi_{+,\varepsilon}}[ A_{2,\varepsilon} v_\varepsilon -iMv_\varepsilon]\right)(s,x-t+s) \, ds,
  \\
  (e^{-i\phi_{-,\varepsilon}} v_\varepsilon)(t,x) &= \int_0^t \left( e^{-i\phi_{-,\varepsilon}}[- A_{2,\varepsilon} u_\varepsilon -iMu_\varepsilon] \right)(s,x+t-s) \, ds.
\end{align*}
Take $L^1$ norms in $x$ to get
\begin{align*}
  \norm{u_\varepsilon(t)}_{L^1} &\le \norm{\chi f_\varepsilon}_{L^1} + \int_0^t (M + \norm{A_{2,\varepsilon}(s)}_{L^\infty})\norm{v_\varepsilon(s)}_{L^1} \, ds,
  \\
  \norm{v_\varepsilon(t)}_{L^1} &\le \int_0^t (M + \norm{A_{2,\varepsilon}(s)}_{L^\infty})\norm{u_\varepsilon(s)}_{L^1} \, ds.
\end{align*}
Adding these and applying Gr\"onwall's inequality yields
\begin{equation}\label{uvBound}
  \norm{u_\varepsilon(t)}_{L^1} + \norm{v_\varepsilon(t)}_{L^1} \le \norm{\chi f_\varepsilon}_{L^1} e^{\int_0^t (M + \norm{A_{2,\varepsilon}(s)}_{L^\infty}) \, ds}.
\end{equation}
Observing that
\[
  \norm{\chi f_\varepsilon}_{L^1} \le C := \int_{\R} \frac{\abs{\chi(x)}}{\abs{x}^{1/2}} \, dx < \infty,
\]
and defining the continuous function $g_\varepsilon \colon [0,\infty) \to [0,\infty)$ by
\[
  g_\varepsilon(t) = \sup_{0 \le s \le t} \norm{A_{2,\varepsilon}(s)}_{L^\infty},
\]
we conclude from \eqref{A2Est0} and \eqref{uvBound} that
\begin{equation}\label{gEst}
  g_\varepsilon(t) \le C^2\left(1 + t (M + g_\varepsilon(t)) e^{t (M + g_\varepsilon(t))}\right)
  \left(t (M + g_\varepsilon(t)) e^{t (M + g_\varepsilon(t))}\right).
\end{equation}

We now use a boot-strap argument to show that there exists a $\delta > 0$, depending only on $C$ and $M$, such that for $0 \le t \le \delta$,
\begin{equation}\label{gBound}
  g_\varepsilon(t) \le 1.
\end{equation}
Assuming this holds for some $t > 0$, then by \eqref{gEst} we have
\begin{equation}\label{gBound2}
  g_\varepsilon(t) \le C^2 \alpha(t),
\end{equation}
where the increasing function $\alpha \colon [0,\infty) \to [0,\infty)$ is defined by
\[
  \alpha(t) = \left(1 + t (M + 1) e^{t (M + 1)}\right)
  \left(t (M + 1) e^{t (M + 1)}\right).
\]
Since $\alpha(0) = 0$, there exists $\delta > 0$, depending only on $M$ and $C$, such that
\begin{equation}\label{deltaDef}
  C^2 \alpha(\delta) \le \frac12.
\end{equation}

By a continuity argument it now follows that \eqref{gBound} holds for all $t \in [0,\delta]$. Indeed, since $g_\varepsilon(0) = 0$, then \eqref{gBound} certainly holds for sufficiently small $t > 0$. And if \eqref{gBound} holds on some interval $[0,T] \subset [0,\delta]$, then by \eqref{gBound2} and \eqref{deltaDef} we have in fact $g_\varepsilon(t) \le 1/2$ on that interval, so \eqref{gBound} holds on a slightly larger interval.

This concludes the proof of \eqref{Claim1} for $d=2$. For $d=3$ the same proof goes through with some obvious changes. Indeed, the system \eqref{MD3d} has essentially the same structure as \eqref{MD2d}, and in particular the equations for $A_2, A_3$ have null forms in the right hand side. Thus, we obtain
\begin{equation}\label{Claim1'3d}
  \sup_{(t,x) \in [0,T] \times \R} \left( \abs{A_{2,\varepsilon}(t,x)} + \abs{A_{3,\varepsilon}(t,x)} \right) \le 1
\end{equation}
for $T > 0$ sufficiently small.

\subsection{Proof of \eqref{Claim2}}

Since we restrict to $x_2=\dots=x_d=0$ and write $x=x_1$, then \eqref{Claim2} reduces to proving that for $T > 0$ sufficiently small,
\begin{equation}\label{Claim2'}
  \abs{u_\varepsilon(t,x)}^2 \ge \frac12 \abs{f_\varepsilon(x-t)}^2
  \quad \text{for $0 < t < x < 1-t$ and $t < T$}.
\end{equation}
We do the proof for $d=2$, and comment on $d=1$ and $d=3$ at the end.

Assuming $d=2$, we use \eqref{MD2dSquared}. Thus,
\begin{align*}
  (\partial_t + \partial_x) \abs{u_\varepsilon}^2 &= F_\varepsilon := 2A_{2,\varepsilon} \re (u_\varepsilon \overline {v_\varepsilon}) - 2M \im (u_\varepsilon \overline {v_\varepsilon}),
  \\
  (\partial_t - \partial_x) \abs{v_\varepsilon}^2 &= G_\varepsilon:= - 2A_{2,\varepsilon} \re (u_\varepsilon \overline {v_\varepsilon}) + 2M \im (u_\varepsilon \overline {v_\varepsilon}),
\end{align*}
and therefore
\begin{align}
  \label{u}
  \abs{u_\varepsilon(t,x)}^2 &= \abs{\chi f_\varepsilon(x-t)}^2 + \int_0^t F_\varepsilon(s,x-t+s) \, ds,
  \\
  \label{v}
  \abs{v_\varepsilon(t,x)}^2 &= \int_0^t G_\varepsilon(s,x+t-s) \, ds.
\end{align}
By \eqref{Claim1'}, for $T > 0$ sufficiently small we have
\begin{equation}\label{FGbounds}
  \abs{F_\varepsilon},\abs{G_\varepsilon} \le (M+1)(\abs{u_\varepsilon}^2+\abs{v_\varepsilon}^2) \quad \text{in $[0,T] \times \R$},
\end{equation}
hence
\begin{align}
  \label{u2}
  \abs{u_\varepsilon(t,x)}^2 &\le \abs{\chi f_\varepsilon(x-t)}^2 + (M+1)\int_0^t (\abs{u_\varepsilon}^2+\abs{v_\varepsilon}^2)(s,x-t+s) \, ds,
  \\
  \label{v2}
  \abs{v_\varepsilon(t,x)}^2 &\le (M+1) \int_0^t (\abs{u_\varepsilon}^2+\abs{v_\varepsilon}^2)(s,x+t-s) \, ds
\end{align}
for $t \in [0,T]$, $x \in \R$.

The idea is now to apply a boot-strap argument. For $\rho \in (0,1-2T)$, define
\[
  B_{\rho,\varepsilon}(s) = \sup_{\rho + s \le y \le 1-s} \left( \abs{u_\varepsilon(s,y}^2+\abs{v_\varepsilon(s,y)}^2 \right)
  \qquad (0 \le s \le T).
\]
If $\rho + t \le x \le 1-t$,
the integrands in \eqref{u2}, \eqref{v2} are bounded by $B_{\rho,\varepsilon}(s)$, so
\[
  \abs{u_\varepsilon(t,x)}^2 + \abs{v_\varepsilon(t,x)}^2
  \le \frac{1}{\sqrt{\varepsilon^2 + (x-t)^2}} + 2(M+1)\int_0^t B_{\rho,\varepsilon}(s) \, ds
\]
Taking the supremum over $x \in [\rho+t,1-t]$ gives
\[
  B_{\rho,\varepsilon}(t)
  \le \frac{1}{\sqrt{\varepsilon^2 + \rho^2}} + 2(M+1)\int_0^t B_{\rho,\varepsilon}(s) \, ds.
\]
By Gr\"onwall's inequality we conclude that
\begin{equation}\label{Bbound}
  B_{\rho,\varepsilon}(t)
  \le \frac{1}{\sqrt{\varepsilon^2 + \rho^2}} e^{2(M+1)t} \le \frac{3}{\sqrt{\varepsilon^2 + \rho^2}} \quad \text{for $t \in [0,T]$},
\end{equation}
assuming $T > 0$ is so small that $2(M+1)T < 1$.

Combining \eqref{Bbound}, \eqref{u} and \eqref{FGbounds}, we obtain, for $\rho > 0$, $x \in [\rho+t,1-t]$ and $t \le T$,
\begin{align*}
  \abs{u_\varepsilon(t,x)}^2 
  &\ge \abs{\chi f_\varepsilon(x-t)}^2 - (M+1)\int_0^t (\abs{u_\varepsilon}^2+\abs{v_\varepsilon}^2)(s,x-t+s) \, ds
  \\
  &\ge \frac{1}{\sqrt{\varepsilon^2 + (x-t)^2}} - (M+1) \int_0^t B_{\rho,\varepsilon}(s) \, ds
  \\
  &\ge \frac{1}{\sqrt{\varepsilon^2 + (x-t)^2}} -  \frac{3(M+1)t}{\sqrt{\varepsilon^2 + \rho^2}},
\end{align*}
where we also used the fact that $\chi=1$ on $[-1,1]$.
Choosing $\rho = x-t$ and assuming $T > 0$ so small that $6(M+1)T < 1$, we obtain the claimed inequality \eqref{Claim2'}.

This completes the proof of \eqref{Claim2} for $d=2$. The proof for $d=1,3$ works out the same way, but instead of \eqref{MD2dSquared} we use either \eqref{MD1dSquared} or \eqref{MD3dSquared}, and in the case $d=3$ we use \eqref{Claim1'3d} instead of \eqref{Claim1'}.

\subsection{Proof of Theorem \ref{MainThm3}}

Define $\psi_0^{\mathrm{bad}} \in X_0 \setminus L^2(\R^d)$ as in section \ref{DataSection}. Assume there exist (i) a neighbourhood $\Omega_0$ of $\psi_0^{\mathrm{bad}}$ in $X_0$, and (ii) continuous maps
\[
  I_\mu, J_\mu \colon \Omega_0 \longrightarrow \mathcal D'(B),
\]
such that if $\psi_0 \in \Omega_0 \cap C_c^\infty(\R^d)$, then defining
\[
  a_\mu = I_\mu[\psi_0], \qquad b_\mu = J_\mu[\psi_0] \qquad (\mu=0,\dots,d),
\]
the constraint equations \eqref{Constraints} are satisfied in $B$.

We will show that these assumptions lead to a contradiction. Define $\psi_{0,\varepsilon} \in C_c^\infty(\R^d)$ as in section \ref{DataSection}. Then $\psi_{0,\varepsilon} \to \psi_0^{\mathrm{bad}}$ in $X_0$ as $\varepsilon \to 0$, so $\psi_{0,\varepsilon}$ belongs to $\Omega_0$ for all $\varepsilon > 0$ small enough, and we may define
\[
  a_{\mu,\varepsilon} = I_\mu[\psi_{0,\varepsilon}], \qquad b_{\mu,\varepsilon} = J_\mu[\psi_{0,\varepsilon}] \qquad (\mu=0,\dots,d).
\]
By assumption, these fields satisfy the constraints \eqref{Constraints} in $B$, so in particular
\[
  \sum_{j=1}^d \partial_j (\partial_j a_{0,\varepsilon} - b_{j,\varepsilon}) = \abs{\psi_{0,\varepsilon}}^2 \quad \text{in $B$}.
\]
By the assumed continuity of the maps $I_\mu, J_\mu$, the left hand side must converge in $\mathcal D'(B)$ as $\varepsilon \to 0$. But the right hand side equals
\[
  \abs{\psi_{0,\varepsilon}(x)}^2 = \frac{1}{\sqrt{\varepsilon^2 + x_1^2}}
  \quad \text{for $x \in B$},
\]
and this function does not have a limit in the sense distributions in $B$, as $\varepsilon \to 0$.

\section{Proof of well-posedness}\label{WPproof}

In this section we prove Theorem \ref{WPThm}. To ease the notation we write $x=x_1$ throughout. To prove local existence we use an iteration and rely only on the following elementary estimates.

\subsection{Linear estimates}

Firstly, for the Dirac equation
\[
  (- i\gamma^0 \partial_t - i\gamma^1 \partial_x + M) \psi = F(t,x), \qquad \psi(0,x) = \psi_0(x),
\]
we shall use the energy inequality, for $t > 0$,
\begin{equation}\label{EnergyInequality}
  \norm{\psi(t)}_{L^2(\R)} \le \norm{\psi_0}_{L^2(\R)} + \int_0^t \norm{F(s)}_{L^2(\R)} \, ds.
\end{equation}
This is proved as follows. By approximation, we may assume that $\psi_0$ and $F$ are smooth and compactly supported in $x$. Premultiplying the equation by $i\overline\psi=i\psi^*\gamma^0$ and taking real parts yields $\partial_t \rho + \partial_x j = \re(i\psi^*\gamma^0 F)$, where $\rho=\psi^*\psi$ and $j=\psi^*\gamma^0\gamma^1\psi$. Integration in $x$ gives
\[
  \frac{d}{dt} \int_{\R} \abs{\psi}^2 \, dx = 2\re \int_{\R} i\psi^* \gamma^0 F \, dx \le 2 \norm{\psi(t)}_{L^2} \norm{F(t)}_{L^2},
\]
which implies \eqref{EnergyInequality}.

Secondly, for the wave equation
\[
  \square u = G(t,x), \qquad u(0,x) = f(x), \qquad \partial_t u(0,x) = g(x),
\]
we shall use the estimates, for $t > 0$,
\begin{align}
  \label{WaveEstimateA}
  \norm{u(t)}_{L^\infty(\R)} &\le \norm{f}_{L^\infty(\R)} + \norm{g}_{L^1(\R)}
  + \int_0^t \norm{G(s)}_{L^1(\R)} \, ds,
  \\
  \label{WaveEstimateB}
  \norm{\partial_x u(t)}_{L^1(\R)} &\le \norm{f'}_{L^1(\R)} + \norm{g}_{L^1(\R)}
  + \int_0^t \norm{G(s)}_{L^1(\R)} \, ds,
  \\
  \label{WaveEstimateC}
  \norm{\partial_t u(t)}_{L^1(\R)} &\le \norm{f'}_{L^1(\R)} + \norm{g}_{L^1(\R)}
  + \int_0^t \norm{G(s)}_{L^1(\R)} \, ds,
\end{align}
which are immediate from D'Alembert's formula,
\[
  u(t,x) = \frac{f(x+t) + f(x-t)}{2} + \frac12 \int_{x-t}^{x+t} g(y) \, dy
  + \frac12 \int_0^t \int_{x-(t-s)}^{x+t-s} G(s,y) \, dy \, ds.
\]
Adding \eqref{WaveEstimateA}--\eqref{WaveEstimateC} gives
\begin{equation}\label{WaveEstimate}
  \norm{u(t)}_{AC} + \norm{\partial_t u(t)}_{L^1}
  \le 3 \left( \norm{f}_{AC(\R)} + \norm{g}_{L^1} + \int_0^t \norm{G(s)}_{L^1} \, ds \right).
\end{equation}

\subsection{The local result}\label{Local}

With the above linear estimates, it is now an easy matter to prove the local well-posedness of \eqref{MD'} by iteration, for data with the regularity $\psi_0 \in L^2(\R)$, $a_\mu \in AC(\R)$ and $b_\mu \in L^1(\R)$. Indeed, applying the energy inequality \eqref{EnergyInequality} to the Dirac equation in \eqref{MD'}, we use the trivial bilinear estimate
\[
  \int_0^T \norm{A_\mu \gamma^\mu \psi(s)}_{L^2} \, ds
  \le CT \norm{A}_{C([0,T];L^\infty)} \norm{\psi}_{C([0,T];L^2)},
\]
where $C([0,T];L^p)$ is equipped with the sup norm. Moreover, applying \eqref{WaveEstimate} to the wave equations in \eqref{MD'} we use the equally trivial bilinear bound
\begin{equation}\label{BilinearB}
  \int_0^T \norm{\psi^* \gamma^0 \gamma^\mu \psi(s)}_{L^1} \, ds
  \le CT \norm{\psi}_{C([0,T];L^2)}^2.
\end{equation}
By a standard contraction argument, which we do not repeat here, one now obtains local well-posedness with a time of existence $T > 0$ determined by 
\[
  CT\left( \norm{\psi_0}_{L^2} + \sum_{\mu=0}^d \norm{a_\mu}_{AC} + \sum_{\mu=0}^d \norm{b_\mu}_{L^1} \right) \le 1,
\]
where $C$ is a universal constant. This proves Theorem \ref{WPThm} for such $T$. Next, we show that the results extend globally.

\subsection{The global result}\label{Global} To extend the local result globally in time, it suffices to obtain an a priori bound on the solution $(\psi,A,\partial_t A)(t)$ in $L^2(\R) \times AC(\R) \times L^1(\R)$. For $\psi$, this bound is directly provided by the conservation of charge, \eqref{1dCharge}. The latter also provides the necessary bound for $(A,\partial_t A)$, via the linear estimate \eqref{WaveEstimate} and the bilinear estimate \eqref{BilinearB}. This concludes the proof of Theorem \ref{WPThm}.

\bibliographystyle{amsplain}
\bibliography{database}

\end{document}